\documentclass[11pt,letterpaper,reqno]{amsart}
\usepackage[T1]{fontenc}
\usepackage{amsmath,amssymb,amsthm}
\usepackage{enumitem}
\usepackage{microtype}
\usepackage{doi}
\usepackage[margin=1in]{geometry}

\usepackage[dvipsnames]{xcolor}
\usepackage{hyperref}
\usepackage{bookmark}

\hypersetup{
colorlinks=true,linkcolor=RoyalBlue,citecolor=ForestGreen!65!black,urlcolor=BrickRed
}

\newtheorem{thm}{Theorem}[section]
\newtheorem{lem}[thm]{Lemma}
\newtheorem{prop}[thm]{Proposition}
\newtheorem{cor}[thm]{Corollary}
\theoremstyle{remark}

\numberwithin{equation}{section}

\newcommand{\T}{\mathbb T}
\newcommand{\D}{\mathbb D}
\newcommand{\C}{\mathbb C}
\newcommand{\R}{\mathbb R}
\newcommand{\norm}[1]{\left\lVert #1\right\rVert}
\newcommand{\dd}{\,\mathrm d}
\newcommand{\e}{\mathrm e}
\newcommand{\Rea}{\operatorname{Re}}
\newcommand{\PV}{\operatorname{PV}}
\newcommand{\Ent}{\operatorname{Ent}}
\newcommand{\Var}{\operatorname{Var}}
\newcommand{\Dop}{\mathcal D}

\begin{document}

\title[Baernstein quasi-norm monotonicity]
{Baernstein's Quasi-Norm Monotonicity Conjecture for Polynomials with Unimodular Zeros}

\dedicatory{Dedicated to the memory of Professor Albert Baernstein II.}

\author[T.~Zhang]{Teng Zhang}

\address{School of Mathematics and Statistics, Xi'an Jiaotong University, Xi'an 710049, P. R. China}
\email{teng.zhang@stu.xjtu.edu.cn}

\subjclass[2020]{Primary 30A10, 30C10; Secondary 26D15, 30H10, 41A17}

\keywords{Baernstein conjecture; integral means; quasi-norms; self-inversive polynomials; finite Blaschke products; relative entropy; Hardy--Bergman identity; Visser inequality; Erd\H{o}s--Szekeres products}

\begin{abstract}
Let $m$ denote the normalized Haar measure on the unit circle $\T$. For $0<r<\infty$, define $\lVert f\rVert_{r}:=\left(\int_{\T}|f|^{r}\,\dd m\right)^{1/r}$, with $\lVert f\rVert_{0}$ and $\lVert f\rVert_{\infty}$ interpreted as the geometric mean and the supremum norm, respectively.
Set $Q_n(z)=1+z^n$. We prove that, for every nonzero polynomial $p$ of degree $n$ whose zeros all lie on $\T$, \[ \frac{\norm{p}_s}{\norm{Q_n}_s} \le \frac{\norm{p}_t}{\norm{Q_n}_t}, \qquad 0\le s\le t\le\infty. \] This settles Baernstein's quasi-norm monotonicity conjecture. As
corollaries, we obtain an $L^r$ extension of Visser's coefficient inequality, the sharp O'Hara--Rodriguez inequality and its higher-power analogues, the Erd\H{o}s--Szekeres product bound $ \left\|\prod_{j=1}^N(1-z^{s_j})\right\|_\infty\ge2\sqrt N $ for all positive integers $s_1,\ldots,s_N$ and Agler--McCarthy's entropy conjecture. We also provide a Lean 4 formalization of the main results.
\end{abstract}

\maketitle
\hypersetup{pdftitle={Baernstein's Quasi-Norm Monotonicity Conjecture for Polynomials with Unimodular Zeros},
 pdfauthor={Teng Zhang},pdfsubject={Baernstein quasi-norm monotonicity for unimodular-root polynomials}}
\setcounter{tocdepth}{1}
\tableofcontents

\section{Introduction}

\subsection{Baernstein's conjecture and its consequences}

Let $\D=\{z\in\C:|z|<1\}$ and $\T=\partial\D$, and let $m$ denote normalized Haar measure on $\T$, given by
$\dd m(\e^{i\theta})=\dd\theta/(2\pi)$ for $0\le\theta<2\pi$. For a nonzero polynomial $f$ and $0<r<\infty$, write
\[
  \norm{f}_r=\left(\int_{\T}|f|^r\,\dd m\right)^{1/r},
\]
and use the continuous endpoint conventions
\[
  \norm{f}_0=\exp\int_{\T}\log|f|\,\dd m,
  \qquad
  \norm{f}_\infty=\operatorname*{ess\,sup}_{\T}|f|.
\]
Thus $\norm{\cdot}_r$ is a norm for $r\ge1$, a quasi-norm for $0<r<1$, and the Mahler measure when $r=0$.

The principal aim of this paper is to prove the quasi-norm monotonicity conjecture for polynomials with unimodular zeros, posed by A.~Baernstein II in 2008; see \cite[Conjecture~2.2]{AM21}. We settle it in the following form.

\begin{thm}[Baernstein's conjecture]\label{thm:main}
Let $n\ge1$, and let $p$ be a nonzero complex polynomial of degree $n$ all of whose zeros lie on $\T$. For $Q_n(z)=1+z^n$ and $0\le s\le t\le\infty$,
\begin{equation}\label{eq:main}
  \frac{\norm{p}_s}{\norm{Q_n}_s}
  \le
  \frac{\norm{p}_t}{\norm{Q_n}_t}.
\end{equation}
Equivalently, the function
\[
  r\longmapsto \log\frac{\norm{p}_r}{\norm{Q_n}_r}
\]
is nondecreasing on $[0,\infty]$.
\end{thm}

A Lean~4 formalization of Theorem~\ref{thm:main}, including arbitrary
root multiplicities and the endpoint exponents $0$ and $\infty$,
is available in~\cite{Zha26}.

Write $p(z)=\sum_{k=0}^n a_kz^k$ and put $C_r:=\norm{1+z}_r$, with
$C_0=1$ and $C_\infty=2$. Since $z\mapsto z^n$ preserves $m$,
$\norm{Q_n}_r=C_r$. Theorem~\ref{thm:main} therefore gives the sharp comparison
\[
 \norm{p}_t\ge\frac{C_t}{C_s}\norm{p}_s,
 \qquad 0\le s\le t\le\infty.
\]
Equality is attained by $p(z)=c(z^n-\omega)$, where $c\ne0$ and $|\omega|=1$.
Its principal consequences, proved in Section~\ref{sec:applications}, are:
\begin{itemize}[leftmargin=*,itemsep=0.6ex]
\item The case $s=0$ gives the sharp $L^r$ extension
\[
 |a_0|+|a_n|\le\frac2{C_r}\norm{p}_r
\]
of Visser's coefficient inequality \cite[Theorem~3]{Vis45}.
\item The case $(s,t)=(2,\infty)$ gives the O'Hara--Rodriguez inequality
\cite[Corollary~1]{OR74},
\[
 \norm{p}_\infty^2\ge2\sum_{k=0}^n|a_k|^2.
\]
Together with Tang's arithmetic estimate \cite[Theorem~3.2]{Tan26}, this yields
\[
 \left\|\prod_{j=1}^N(1-z^{s_j})\right\|_\infty\ge2\sqrt N
 \qquad(s_j\in\mathbb N,\ s_j\ge1),
\]
recovering his improvement of the Erd\H{o}s--Szekeres bound \cite{ES59}.
\item Differentiation with respect to the exponent gives a relative-entropy
comparison for every $r>0$. At $r=2$, it is the homogeneous entropy
conjecture of Agler and McCarthy \cite[Conjecture~1.4]{AM21}, recently
proved by Lei and Zhang \cite[Theorem~1.3]{LZ26}.
\end{itemize}
Further consequences include sharp higher-moment, mixed derivative, and
weighted power-sum inequalities. The entropy problem is connected with
the Krzy\.z conjecture \cite{Krz68}: in the program of Agler and McCarthy,
the entropy inequality and an additional full-degree condition for extremal
functions together imply that conjecture \cite[Theorem~1.8]{AM21}.

\subsection{Relation to previous work}

The study of integral means of analytic functions can be traced back to Hardy \cite{Har15}, who investigated the dependence of boundary means on the radius. Baernstein \cite{Bae74} later introduced the star-function and circular symmetrization methods, which provide a powerful framework for sharp comparisons of integral means.
Theorem~\ref{thm:main} is related to the classical theory of multiplier inequalities for polynomials. The Schur--Szeg\H{o} composition theorem of de Bruijn and Springer \cite{BS47} gives sharp inequalities for suitable polynomial compositions. Arestov \cite{Are82,Are90} extended their results to a broad class of integral functionals, including quasi-norms, and Pritsker \cite{Pri17} formulated related sharp integral-norm inequalities in terms of polynomial multipliers. These results can imply the endpoint case $s=0$ of Theorem~\ref{thm:main}, but they do not compare two arbitrary positive exponents.

Closely related developments concern sharp derivative estimates under
restrictions on the zeros. Motivated by a conjecture of Erd\H{o}s
\cite{Erd40}, Lax \cite{Lax44} proved that every polynomial $p$ of
degree $n$ having no zeros in $\D$ satisfies
\[
\norm{p'}_\infty\le \frac n2\norm{p}_\infty,
\]
and the constant $n/2$ is sharp. Ankeny--Rivlin \cite{AR55},
Boas--Rahman \cite{BR62}, and Rahman--Schmeisser
\cite{RS88,RS02} subsequently developed related pointwise and
integral-norm inequalities under various restrictions on the zeros.
These results compare a polynomial with its derivative at a fixed
exponent, whereas Theorem~\ref{thm:main} compares the same polynomial
at different exponents.

Two neighboring extremal problems further clarify the scope of the
present result. First, the Erd\H{o}s--Tur\'an theorem \cite{ET50} controls the
angular discrepancy of the zeros in terms of a quantitative size
parameter of the polynomial. By
contrast, the endpoint case of Theorem~\ref{thm:main} yields only the
distribution-free estimate
\[
\frac{\norm{p}_\infty}{|a_n|}\ge 2.
\]
Thus any refinement incorporating a stability term that measures the
discrepancy of the zero-counting measure from normalized Haar measure
would require genuinely new information beyond a further
specialization of the exponents in Theorem~\ref{thm:main}.
Second, Erd\H{o}s, Herzog, and Piranian \cite[p.~142, Problem~12]{EHP58} asked whether, among monic
polynomials of fixed degree, the lemniscate
$
\{z:|p(z)|=1\}
$
has maximal length for $p(z)=z^n-1$. Tao \cite[Theorem~1.1(iv)]{Tao25} recently proved this
conjecture for all sufficiently large degrees. If this lemniscate contains no
critical point of $p$, its inverse branches give
\[
\operatorname{length}\{z:|p(z)|=1\}
=
\int_0^{2\pi}
\sum_{p(z)=\e^{it}}
\frac{\dd t}{|p'(z)|}.
\]
This functional depends both on derivatives and on a moving planar
level set, whereas Theorem~\ref{thm:main} controls integral means of
$|p|$ on the fixed circle $\T$. 

A more recent connection arises from the homogeneous entropy conjecture of Agler and McCarthy \cite{AM21}, which was introduced in their study of the Krzy\.z conjecture. They \cite[Section~10]{AM21} observed that Baernstein's conjecture implies the corresponding entropy inequality at the exponent $r=2$. Lei and Zhang  \cite[Theorem~1.3]{LZ26} recently proved this inequality and determined its equality cases.

\subsection{Proof strategy}

The proof has four steps.
\begin{enumerate}[label=\textnormal{(\arabic*)},leftmargin=*,itemsep=0.8ex]
\item \emph{Polar factorization and entropy.}
After a rotation and scaling, assume $p=p^*$, $p(0)=1$, and that $p$ is monic.
For simple zeros, the polar factor
$q=p-zp'/n$ is zero-free on $\overline\D$, and
$B=q^*/q$ is a finite Blaschke product with $B(0)=0$ and $B_*m=m$.
The factorization $p=q(1+B)$ and an entropy identity reduce the theorem
to the nonnegativity of a centered integral $\mathcal P_r(q,B)$
(Section~\ref{sec:reduction}).
\item \emph{A universal kernel.}
Writing $r=2a$ and $g=q^a$, Fourier expansion and a Dirichlet recurrence give
\[
 \mathcal P_{2a}(q,B)
 =\frac{M_{Q_n}(2a)}{an}
   \int_{\D}\Rea H_a(B(z))\,|g'(z)|^2\,\dd A(z),
\]
where $\dd A=\dd x\,\dd y/\pi$ and $H_a$ is independent of $p$ and $n$
(Section~\ref{sec:area}). Thus it suffices to prove $\Rea H_a>0$ in $\D$.
\item \emph{Kernel positivity.}
For $0<a\le1$, a sharp summable bound on the Taylor coefficients proves
positivity directly. For $a\ge1$, a Stieltjes representation reduces it
to a real principal-value transform. A positive Abel decomposition and
a finite-basis inequality prove that this transform is positive
(Sections~\ref{sec:small} and~\ref{sec:large}).
\item \emph{Limits.}
The entropy derivative now proves monotonicity for simple roots and
positive finite exponents. Perturbing the root arguments while preserving
their sum removes multiplicities. The limits at $0$ and $\infty$, followed
by undoing the normalization, complete the proof
(Section~\ref{sec:completion}).
\end{enumerate}
Section~\ref{sec:applications} gives the consequences stated above.

\medskip\noindent\textbf{Acknowledgements and AI tools disclosure. }
The authors are deeply grateful to Professor John McCarthy for his encouragement.

This work was supported by the China Scholarship Council, the Young Elite Scientists Sponsorship Program for PhD Students (China Association for Science and Technology), and the Fundamental Research Funds for the Central Universities at Xi'an Jiaotong University (Grant No.~xzy022024045).

An earlier version of this paper, in which all proofs had already been completed, did not involve the use of any AI tools. In the current version, we used ChatGPT for proofreading, correcting grammatical errors, and simplifying the presentation of the proofs. The overall ideas and arguments remain those developed by the authors in the earlier version.

The Lean~4 formalization in~\cite{Zha26} was generated using
OpenAI's Codex, and its proofs were checked by the Lean kernel.

\section{Normalization and entropy reduction}\label{sec:reduction}

This section isolates the two structural ingredients used throughout the proof: a polar factorization for self-inversive polynomials and the entropy identity obtained by differentiating an $L^r$ mean. The objects introduced here recur in the sequel; the auxiliary notation used in later sections will be defined locally.

Fix an integer $n\ge1$. For every polynomial $f$ of degree at most $n$, its reflection relative to the fixed degree $n$ is
\begin{equation}\label{eq:reflection}
  f^*(z):=z^n\overline{f(1/\bar z)}.
\end{equation}
We also write
\begin{equation}\label{eq:euler-operator}
  \Dop f(z):=zf'(z)
\end{equation}
for the Euler operator.

\subsection{Self-inversive normalization and the polar factor}

We begin by removing the inessential scalar and rotational degrees of freedom. The same ideas also appear in \cite{LZ26}.

\begin{lem}\label{lem:normalization}
Let $P$ be a nonzero polynomial of degree $n$ whose zeros lie on $\T$. Then there are $c\ne0$ and $\eta\in\T$ such that
\begin{equation}\label{eq:normalized-p}
  p(z):=cP(\eta z)
  =1+a_1z+\cdots+a_{n-1}z^{n-1}+z^n
  \quad\text{and}\quad p=p^*,
\end{equation}
where $p^*$ is defined by \eqref{eq:reflection}. Moreover,
\begin{equation}\label{eq:normalization-norms}
  \norm{P}_r=|c|^{-1}\norm{p}_r,
  \qquad 0\le r\le\infty.
\end{equation}
Finally, for $Q_n(z)=1+z^n$,
\begin{equation}\label{eq:model-norm}
  \norm{Q_n}_r
  =2\left(
    \frac{\Gamma((r+1)/2)}{\sqrt\pi\,\Gamma(1+r/2)}
  \right)^{1/r},
  \qquad 0<r<\infty,
\end{equation}
while $\norm{Q_n}_0=1$ and $\norm{Q_n}_\infty=2$.
\end{lem}

\begin{proof}
Write $P(z)=c_n\prod_{j=1}^n(z-\zeta_j)$, with $|\zeta_j|=1$, and let
$c_0=P(0)$. Then $\alpha:=c_0/c_n=(-1)^n\prod_j\zeta_j\in\T$.
Choose $\eta\in\T$ with $\eta^n=\alpha$ and set $c=c_0^{-1}$.
The polynomial $p(z)=cP(\eta z)$ is monic and has constant coefficient $1$.
For a unit root $\zeta$,
\[
 1-\overline\zeta z=-\overline\zeta(z-\zeta).
\]
Multiplying over the roots, with multiplicity, gives
$p^*=\overline{p(0)}p=p$. Rotation invariance and scalar homogeneity give
\eqref{eq:normalization-norms}, including both endpoints.

Since $z\mapsto z^n$ preserves $m$, one has $\norm{Q_n}_r=\norm{1+z}_r$.
The substitution $z=\e^{i\theta}$ and the beta integral give
\eqref{eq:model-norm}. Jensen's formula gives $\norm{Q_n}_0=1$, and
$\max_{\T}|1+z^n|=2$.
\end{proof}

For the rest of Sections~\ref{sec:reduction}--\ref{sec:large}, $p$ denotes a polynomial satisfying \eqref{eq:normalized-p}. Until Section~\ref{sec:completion}, we also assume that the zeros of $p$ are simple. The following polar factorization is the point at which this hypothesis enters; the same idea also appears in \cite{LZ26}.

\begin{prop}\label{prop:polar}
Let $p$ satisfy \eqref{eq:normalized-p} and have simple zeros on $\T$. Define
\begin{equation}\label{eq:qB-definition}
  q:=p-\frac1n\Dop p,
  \qquad
  B:=\frac{q^*}{q},
\end{equation}
where $\Dop$ and the reflection are given by \eqref{eq:euler-operator} and \eqref{eq:reflection}. Then $q$ has no zeros on $\overline\D$, $B$ is a finite Blaschke product with $B(0)=0$, and
\begin{equation}\label{eq:polar-factorization}
  q^*=\frac1n\Dop p,
  \qquad
  p=q+q^*=q(1+B).
\end{equation}
Moreover, $B$ preserves normalized Haar measure:
\begin{equation}\label{eq:B-preserves-m}
  \int_{\T}\varphi(B(z))\,\dd m(z)
  =\int_{\T}\varphi(\zeta)\,\dd m(\zeta)
\end{equation}
for every $\varphi\in L^1(\T,m)$.
\end{prop}

\begin{proof}
Comparing coefficients in $p=p^*$ gives $q^*=zp'/n$ and $p=q+q^*$.
To prove zero exclusion, write the roots of $p$ as $\zeta_1,\ldots,\zeta_n$.
For $|z|\le1$ and $z\ne\zeta_j$,
\[
 2\Rea\frac{z}{z-\zeta_j}
 =1+\frac{|z|^2-1}{|z-\zeta_j|^2}\le1.
\]
If $p(z)\ne0$, its logarithmic derivative therefore gives
\[
 \Rea\frac{q(z)}{p(z)}
 =1-\frac1n\sum_{j=1}^n\Rea\frac{z}{z-\zeta_j}\ge\frac12.
\]
If $p(z)=0$, simplicity and $|z|=1$ give $q(z)=-zp'(z)/n\ne0$.
Thus $q$ is zero-free on $\overline\D$.

Let $d=\deg q<n$ and let $u_1,\ldots,u_d$ be its roots. Since $|u_j|>1$,
factoring $q^*/q$ gives
\[
 B(z)=\lambda z^{n-d}\prod_{j=1}^d
 \frac{z-\alpha_j}{1-\overline{\alpha_j}z},
 \qquad \alpha_j=\frac1{\overline{u_j}},\quad |\lambda|=1.
\]
Hence $B$ is a finite Blaschke product and $B(0)=0$.
In particular, $|B|=1$ on $\T$, and Schwarz's lemma gives $|B(z)|\le|z|$ in $\D$.

For every $k\ge1$, the mean-value property gives $\int_{\T}B^k\,\dd m=0$.
Conjugation gives the same identity for negative powers. Thus $B_*m$ and
$m$ have identical Fourier coefficients and are equal, proving
\eqref{eq:B-preserves-m}; see also \cite[Lemma~1]{Nor68}.
\end{proof}

\subsection{Differentiation in the exponent}

For a nonzero polynomial $f$ and $r>0$, define its $r$-moment and the associated probability measure by
\begin{equation}\label{eq:moment-density}
  M_f(r):=\int_{\T}|f|^r\,\dd m,
  \qquad
  \dd\rho_{f,r}:=\frac{|f|^r}{M_f(r)}\,\dd m.
\end{equation}
For a probability measure $\rho\ll m$, its relative entropy with respect to $m$ is
\begin{equation}\label{eq:relative-entropy}
  \mathsf D(\rho\Vert m)
  :=\int_{\T}\log\left(\frac{\dd\rho}{\dd m}\right)\,\dd\rho,
\end{equation}
with the convention $0\log0=0$. The next identity shows that differentiation of an $L^r$-mean produces this entropy.

\begin{lem}\label{lem:entropy-derivative}
Let $f$ be a nonzero polynomial, and let $M_f(r)$ and $\rho_{f,r}$ be defined by \eqref{eq:moment-density}. Then, for every $r>0$,
\begin{equation}\label{eq:entropy-derivative}
  r^2\frac{\dd}{\dd r}\log\norm{f}_r
  =r\frac{M_f'(r)}{M_f(r)}-\log M_f(r)
  =\mathsf D(\rho_{f,r}\Vert m).
\end{equation}
In particular, let $p$ be the normalized polynomial in \eqref{eq:normalized-p}, and set $Q_n(z):=1+z^n$. If
\begin{equation}\label{eq:entropy-comparison}
  \mathsf D(\rho_{p,r}\Vert m)
  \ge
  \mathsf D(\rho_{Q_n,r}\Vert m)
\end{equation}
holds for every $r>0$, then
$r\mapsto\log(\norm{p}_r/\norm{Q_n}_r)$ is nondecreasing on $(0,\infty)$.
\end{lem}

\begin{proof}
Fix $r_0>0$ and let $K=[r_0/2,3r_0/2]$. On
$K\times[0,\norm{f}_\infty]$, the function $x^u\log x$, assigned the value
$0$ at $x=0$, is continuous and bounded. It therefore supplies a uniform
integrable bound for differentiation in $u$. Consequently,
\[
 M_f'(r)=\int_{\T}|f|^r\log|f|\,\dd m.
\]
Differentiating $\log\norm{f}_r=r^{-1}\log M_f(r)$ gives the first equality
in \eqref{eq:entropy-derivative}; substitution of the density
$|f|^r/M_f(r)$ gives the second. Subtracting the formulas for $p$ and $Q_n$
yields
\[
 \frac{\dd}{\dd r}\log\frac{\norm{p}_r}{\norm{Q_n}_r}
 =\frac{\mathsf D(\rho_{p,r}\Vert m)-\mathsf D(\rho_{Q_n,r}\Vert m)}{r^2},
\]
which proves the final assertion.
\end{proof}

\subsection{The centered polar term}

Let $p,q,B$ be as in Proposition~\ref{prop:polar}. For $r>0$, define
\begin{equation}\label{eq:model-pullback-measure}
  \dd\mu_r
  :=\frac{|1+B|^r}{M_{Q_n}(r)}\,\dd m,
  \qquad
  \ell_r:=\int_{\T}\log|1+B|\,\dd\mu_r.
\end{equation}
Proposition~\ref{prop:polar} implies that $\mu_r$ is a probability measure and that $\ell_r$ depends only on $r$, not on $p$. We also introduce the centered polar functional
\begin{equation}\label{eq:polar-functional}
  \mathcal P_r(q,B)
  :=\int_{\T}|q|^r|1+B|^r
  \bigl(\log|1+B|-\ell_r\bigr)\,\dd m.
\end{equation}
For a probability measure $\mu$ and a nonnegative integrable function $w$, define
\begin{equation}\label{eq:entropy-functional}
  \Ent_\mu(w)
  :=\int w\log w\,\dd\mu
  -\left(\int w\,\dd\mu\right)
   \log\left(\int w\,\dd\mu\right).
\end{equation}

The next decomposition separates an automatically nonnegative entropy term from the only term whose sign remains to be proved.

\begin{prop}\label{prop:entropy-reduction}
Let $p,q,B$ be as in Proposition~\ref{prop:polar}. For every $r>0$,
\begin{equation}\label{eq:entropy-splitting}
  \mathsf D(\rho_{p,r}\Vert m)
  -\mathsf D(\rho_{Q_n,r}\Vert m)
  =\frac{M_{Q_n}(r)}{M_p(r)}
    \Ent_{\mu_r}(|q|^r) 
+\frac{r}{M_p(r)}\,\mathcal P_r(q,B).
\end{equation}
In particular, the entropy comparison \eqref{eq:entropy-comparison} follows once
\begin{equation}\label{eq:polar-target}
  \mathcal P_r(q,B)\ge0
\end{equation}
has been proved.
\end{prop}

\begin{proof}
Put $h=|1+B|^r/M_{Q_n}(r)$, $w=|q|^r$, and
$\Lambda=\int w\,\dd\mu_r=M_p(r)/M_{Q_n}(r)$. The densities of $\mu_r$ and
$\rho_{p,r}$ with respect to $m$ are $h$ and $wh/\Lambda$, respectively.
Since both $B$ and $z\mapsto z^n$ preserve $m$,
\[
 \mathsf D(\rho_{Q_n,r}\Vert m)=\int_{\T}h\log h\,\dd m.
\]
All integrands are integrable: $q$ is bounded above and away from zero,
and $x^r\log x$ extends continuously at zero. Direct subtraction gives
\begin{align*}
 \mathsf D(\rho_{p,r}\Vert m)-\mathsf D(\rho_{Q_n,r}\Vert m)
 &=\frac1\Lambda\int w\log\frac w\Lambda\,\dd\mu_r
   +\int\left(\frac w\Lambda-1\right)\log h\,\dd\mu_r\\
 &=\frac1\Lambda\Ent_{\mu_r}(w)
   +\frac r\Lambda\int w(\log|1+B|-\ell_r)\,\dd\mu_r.
\end{align*}
In the second equality, the constant term in $\log h$ disappears because
$\int(w/\Lambda-1)\,\dd\mu_r=0$; centering uses
$\int\log|1+B|\,\dd\mu_r=\ell_r$. This is \eqref{eq:entropy-splitting}.
Finally, Jensen's inequality for $x\log x$ gives $\Ent_{\mu_r}(w)\ge0$.
\end{proof}

The remainder of the proof is therefore devoted to \eqref{eq:polar-target}.
\section{A boundary-to-area identity}\label{sec:area}

Fix $a>0$ and write $r=2a$. For the polar factor $q$ of
Proposition~\ref{prop:polar}, list its roots as $u_1,\ldots,u_d$.
Since $q(0)=1$ and $|u_j|>1$,
\[
 q(z)=\prod_{j=1}^d(1-z/u_j),\qquad
 L(z)=\sum_{j=1}^d\operatorname{Log}(1-z/u_j).
\]
Each logarithm is analytic near $\overline\D$, because
$\Rea(1-z/u_j)>0$ there. Thus $L(0)=0$, $\e^L=q$, and we may define
\begin{equation}\label{eq:g-definition}
 g=q^a:=\e^{aL},\qquad g'=agq'/q,\qquad |g|^2=|q|^{2a}.
\end{equation}
We use normalized area measure
\begin{equation}\label{eq:area-measure}
 \dd A(z):=\frac{\dd x\,\dd y}{\pi},\qquad z=x+iy.
\end{equation}
The goal of this section is to express $\mathcal P_{2a}(q,B)$ as a
weighted Dirichlet integral.

\subsection{The model Fourier coefficients}

Since $z\mapsto z^n$ preserves the normalized Haar measure $m$,
$M_{Q_n}(2a)=\int_{\T}|1+\zeta|^{2a}\,\dd m(\zeta)$. Define the probability density
\begin{equation}\label{eq:model-density}
  h_a(\zeta)
  :=\frac{|1+\zeta|^{2a}}{M_{Q_n}(2a)},
  \qquad \zeta\in\T,
\end{equation}
and its Fourier coefficients
\begin{equation}\label{eq:alpha-definition}
  \alpha_k(a)
  :=\int_{\T}h_a(\zeta)\zeta^{-k}\,\dd m(\zeta),
  \qquad k\in\mathbb Z.
\end{equation}
A dot will always mean differentiation with respect to $a$; thus
$\dot\alpha_k(a)=\partial_a\alpha_k(a)$. The relevant Fourier coefficients are explicit and have a summable derivative.

\begin{lem}\label{lem:model-coefficients}
For $a>0$, the Fourier coefficients in \eqref{eq:alpha-definition} satisfy
$\alpha_0(a)=1$ and $\alpha_{-k}(a)=\alpha_k(a)\in\R$. For $k\ge1$,
\begin{equation}\label{eq:alpha-formula}
 \alpha_k(a)=\prod_{j=0}^{k-1}\frac{a-j}{a+j+1}
 =\frac{\Gamma(a+1)^2}{\Gamma(a-k+1)\Gamma(a+k+1)},
\end{equation}
where the reciprocal gamma function gives the interpretation at integer $a$.
Moreover,
\begin{equation}\label{eq:alpha-summable}
 \sum_{k\ge1}\bigl(|\alpha_k(a)|+|\dot\alpha_k(a)|\bigr)<\infty.
\end{equation}
\end{lem}

\begin{proof}
The density $h_a(\e^{i\theta})$ is real and even, which gives the symmetry.
With
\[
 I_k(a)=\int_{-\pi/2}^{\pi/2}\cos^{2a}t\cos(2kt)\,\dd t,
\]
the substitution $\theta=2t$ gives $\alpha_k=I_k/I_0$. The function
$F_k(t)=\cos^{2a+1}t\sin((2k+1)t)$ vanishes at both endpoints and satisfies
\[
 F_k'(t)=(k-a)\cos^{2a}t\cos(2kt)
 +(a+k+1)\cos^{2a}t\cos(2(k+1)t).
\]
Integrating gives $(a+k+1)I_{k+1}=(a-k)I_k$. Since $I_0>0$, induction
proves the product formula; the gamma form follows from $\Gamma(x+1)=x\Gamma(x)$.

For summability, write $A_k=|\alpha_k(a)|$, $D_k=|\dot\alpha_k(a)|$, and
choose an integer $N\ge a$. Differentiating the finite-product recurrence gives,
for $k\ge N$,
\[
 A_{k+1}=\frac{k-a}{k+a+1}A_k,\qquad
 D_{k+1}\le\frac{k-a}{k+a+1}D_k+
 \frac{2k+1}{(k+a+1)^2}A_k.
\]
The elementary bounds
\[
 k-(k+1)\frac{k-a}{k+a+1}\ge a,\qquad
 \frac{(k+1)(2k+1)}{(k+a+1)^2}\le2
\]
imply
\[
 aA_k\le kA_k-(k+1)A_{k+1},\qquad
 aD_k\le kD_k-(k+1)D_{k+1}+2A_k.
\]
Summing from $N$ to $N+m-1$ and discarding nonnegative final terms gives
\[
 \sum_{j=0}^{m-1}A_{N+j}\le\frac{NA_N}{a},\qquad
 \sum_{j=0}^{m-1}D_{N+j}\le\frac{ND_N+2NA_N/a}{a}.
\]
The nonnegative partial sums are bounded, proving \eqref{eq:alpha-summable}.
This argument includes integer $a$, even when a coefficient vanishes.
\end{proof}

\subsection{Fourier reduction}

For $k\ge0$, define the boundary correlations
\begin{equation}\label{eq:correlations}
  c_k:=\int_{\T}B^k|g|^2\,\dd m
  =\langle B^k g,g\rangle_{H^2},
\end{equation}
where the $H^2$ inner product is linear in the first variable. Differentiating the model density converts the polar functional into these correlations.

\begin{prop}\label{prop:fourier-reduction}
For every $a>0$, the correlations in \eqref{eq:correlations} satisfy
\begin{equation}\label{eq:fourier-reduction}
  \mathcal P_{2a}(q,B)
  =M_{Q_n}(2a)
   \sum_{k\ge1}\dot\alpha_k(a)\Rea c_k.
\end{equation}
\end{prop}

\begin{proof}
Differentiating the density gives
\[
 \dot h_a(\zeta)=2h_a(\zeta)(\log|1+\zeta|-\ell_{2a}).
\]
The compact domination used in Lemma~\ref{lem:entropy-derivative} applies
also after multiplication by a Fourier character. Hence the Fourier
coefficients of $\dot h_a$ are $\dot\alpha_k(a)$, with zero constant
coefficient. The function $\dot h_a$ is continuous on $\T$, including
$\zeta=-1$, because $x^{2a}\log x\to0$ at zero.
By \eqref{eq:alpha-summable}, the series
\[
 \dot h_a(\zeta)=\sum_{k\ge1}\dot\alpha_k(a)(\zeta^k+\zeta^{-k})
\]
converges absolutely and uniformly; Fourier uniqueness identifies its sum
with $\dot h_a$. Substituting $\zeta=B$, multiplying by
$M_{Q_n}(2a)|g|^2/2$, and integrating is therefore legitimate. Since
$\int B^{-k}|g|^2\,\dd m=\overline{c_k}$, the result is
\eqref{eq:fourier-reduction}.
\end{proof}

\subsection{The Dirichlet recurrence}

For $k\ge0$, define the area moments
\begin{equation}\label{eq:area-moments}
  V_k:=\int_{\D}B^k|g'|^2\,\dd A.
\end{equation}
For functions $u,v$ analytic near $\overline\D$, we use the polarized
Dirichlet identity
\begin{equation}\label{eq:polarized-dirichlet}
 \langle\Dop u,v\rangle_{H^2}=\int_{\D}u'\overline{v'}\,\dd A.
\end{equation}
Indeed, writing $u(z)=\sum_{j\ge0}u_jz^j$ and $v(z)=\sum_{j\ge0}v_jz^j$,
both sides equal $\sum_{j\ge1}j u_j\overline{v_j}$: normalized area gives
$\int_{\D}z^{j-1}\bar z^{k-1}\,\dd A=\delta_{jk}/j$.
The series and their derivatives converge uniformly, so these integrations
are justified. The same coefficient calculation shows that $\Dop$ is
symmetric for the $H^2$ inner product; see also \cite[Section~1.1]{EKMR14}.

\begin{lem}\label{lem:dirichlet-recurrence}
For every $a>0$ and integer $k\ge1$,
\begin{equation}\label{eq:dirichlet-recurrence}
  akn\,c_k
  =(a+k)V_{k-1}-(a-k)V_k.
\end{equation}
\end{lem}

\begin{proof}
Set
\[
  \sigma:=\frac{\Dop q}{q},
  \qquad
  I_k:=\int_{\T}B^k\sigma|g|^2\,\dd m
  \quad(k\ge0).
\]
From $p=q(1+B)$ and $\Dop p=nqB$, we obtain
\begin{equation}\label{eq:B-sigma-identity}
  (1+B)\sigma=nB-\Dop B.
\end{equation}
On $\T$, the identity $B=q^*/q$ gives
\begin{equation}\label{eq:sigma-boundary}
  \overline\sigma=\frac{\sigma}{B}.
\end{equation}
Indeed,
\[
  \frac{\Dop B}{B}=n-\sigma-\overline\sigma
\]
by differentiating $B=q^*/q$ on $\T$, whereas \eqref{eq:B-sigma-identity} gives
$\Dop B/B=n-\sigma-\sigma/B$.

For $k\ge1$, set
\[
  X_k:=\int_{\T}B^{k-1}\Dop B\,|g|^2\,\dd m.
\]
Multiplying \eqref{eq:B-sigma-identity} by $B^{k-1}|g|^2$ gives
\[
  I_{k-1}+I_k=nc_k-X_k.
\]
Since $\Dop g=a\sigma g$ and the Euler operator is symmetric for the $H^2$ inner product on these analytic functions,
\[
  kX_k+aI_k
  =\langle\Dop(B^k g),g\rangle_{H^2}
  =\langle B^k g,\Dop g\rangle_{H^2}
  =aI_{k-1},
\]
where the last equality uses \eqref{eq:sigma-boundary}. Eliminating $X_k$ yields
\begin{equation}\label{eq:I-recurrence}
  knc_k+(a-k)I_k-(a+k)I_{k-1}=0.
\end{equation}

Finally, choose an analytic primitive $u_k$ with
$u_k'=B^k g'$ and $u_k(0)=0$. Then \eqref{eq:polarized-dirichlet} and $\Dop g=a\sigma g$ give
\[
  V_k
  =\int_{\D}B^k|g'|^2\,\dd A
  =\langle\Dop u_k,g\rangle_{H^2}
  =\langle B^k\Dop g,g\rangle_{H^2}
  =aI_k.
\]
Substitution into \eqref{eq:I-recurrence} proves \eqref{eq:dirichlet-recurrence}.
\end{proof}

\subsection{The universal kernel}

Define, for $a>0$ and $z\in\D$,
\begin{equation}\label{eq:H-definition}
  H_a(z)
  :=\sum_{k\ge1}\frac{\dot\alpha_k(a)}{k}
  \left((a+k)z^{k-1}-(a-k)z^k\right).
\end{equation}
Lemma~\ref{lem:model-coefficients} implies absolute and uniform convergence on $\overline\D$. Thus $H_a$ is analytic on $\D$ and continuous on $\overline\D$. Combining the Fourier reduction with the Dirichlet recurrence gives the desired Hardy--Bergman identity.

\begin{prop}\label{prop:area-identity}
For every $a>0$,
\begin{equation}\label{eq:area-identity}
  \mathcal P_{2a}(q,B)
  =\frac{M_{Q_n}(2a)}{an}
   \int_{\D}\Rea H_a(B(z))\,|g'(z)|^2\,\dd A(z).
\end{equation}
\end{prop}

\begin{proof}
Since $|B|\le1$ in $\D$, the moments in
\eqref{eq:area-moments} satisfy $|V_k|\le V_0$. Moreover,
\[
\sum_{k\ge1}\frac{|\dot\alpha_k(a)|}{k}
\bigl((a+k)+|a-k|\bigr)
\le
2(1+a)\sum_{k\ge1}|\dot\alpha_k(a)|
<\infty.
\]
Hence Lemma~\ref{lem:model-coefficients} permits termwise summation.
Multiplying \eqref{eq:dirichlet-recurrence} by
$\dot\alpha_k(a)/(akn)$, summing over $k\ge1$, and taking real parts
gives 
\[
  \sum_{k\ge1}\dot\alpha_k(a)\Rea c_k
  =\frac1{an}
   \int_{\D}\Rea H_a(B(z))\,|g'(z)|^2\,\dd A(z).
\]
Now apply Proposition~\ref{prop:fourier-reduction}.
\end{proof}

The proof of Theorem~\ref{thm:main} has now been reduced to the following universal positivity statement, which no longer involves $p$, $q$, $B$, or the degree $n$.

\begin{thm}\label{thm:kernel}
Let $H_a$ be the analytic function defined by \eqref{eq:H-definition}. Then
\begin{equation}\label{eq:kernel-positivity}
  \Rea H_a(z)>0,
  \qquad a>0,\quad z\in\D.
\end{equation}
\end{thm}

Indeed, Theorem~\ref{thm:kernel} and \eqref{eq:area-identity} imply $\mathcal P_{2a}(q,B)\ge0$. Proposition~\ref{prop:entropy-reduction} then gives the entropy comparison \eqref{eq:entropy-comparison}. Sections~\ref{sec:small} and~\ref{sec:large} prove Theorem~\ref{thm:kernel} in the ranges $0<a\le1$ and $a\ge1$, respectively.
\section{Kernel positivity for \texorpdfstring{$0<a\le1$}{0<a<=1}}\label{sec:small}

Let $H_a$ be the kernel in \eqref{eq:H-definition}. Write
\begin{equation}\label{eq:beta-expansion}
  H_a(z)=\beta_0(a)+\sum_{k\ge1}\beta_k(a)z^k.
\end{equation}
Collecting adjacent terms in \eqref{eq:H-definition} gives
\begin{equation}\label{eq:beta-formula}
  \beta_0(a)=\frac1{a+1},
  \qquad
  \beta_k(a)
  =\frac{a+k+1}{k+1}\dot\alpha_{k+1}(a)
   -\frac{a-k}{k}\dot\alpha_k(a)
  \quad(k\ge1),
\end{equation}
where $\alpha_k(a)$ is defined by \eqref{eq:alpha-definition}. The small-parameter argument rests on the following sharp summable estimate.

\begin{lem}\label{lem:small-coefficients}
Let $0<a\le1$, and let $\beta_k(a)$ be defined by \eqref{eq:beta-expansion}--\eqref{eq:beta-formula}. Then
\begin{equation}\label{eq:small-coefficient-bound}
  |\beta_k(a)|
  \le\frac1{(a+1)k(k+1)},
  \qquad k\ge1.
\end{equation}
\end{lem}

\begin{proof}
First let $0<a<1$. Put $u_k=(-1)^{k-1}\alpha_k(a)>0$ and
$C_k=(-1)^{k-1}k(k+1)\beta_k(a)$. Differentiating
$(a+k+1)\alpha_{k+1}=(a-k)\alpha_k$ in \eqref{eq:beta-formula} gives
\begin{equation}\label{eq:Ck-first}
 C_k=\frac{k(2k+1)}{a+k+1}u_k+(k-a)u_k'.
\end{equation}
Set
\[
 R_k=\prod_{j=1}^k\frac{j-a}{j+a},\qquad
 S_k=\sum_{j=1}^k\frac{2j}{j^2-a^2},\qquad
 T_k=2a+\frac{a(a+1)(2a+1)}{(k-a)(k+a+1)}.
\]
Since $(k-a)u_k=aR_k$ and $R_k'=-R_kS_k$, equation \eqref{eq:Ck-first} becomes
\begin{equation}\label{eq:Ck-exact}
 C_k=R_k(1-aS_k+T_k).
\end{equation}

For the lower bound, we prove $(a+1)aR_kS_k\le1$ by induction.
Write $A_k=(a+1)aR_kS_k$ and $D_k=(a+1)R_k\le1-a<1$.
The initial value is $A_1=2a/(a+1)\le1$. For $j=k+1$, put
$t=(j-a)/(j+a)$ and $s=2j/(j^2-a^2)$. Then
\[
 1-t-ats=\frac{2a^2}{(j+a)^2}\ge0,\qquad
 A_{k+1}=tA_k+D_kats\le t+ats\le1.
\]
Thus \eqref{eq:Ck-exact} gives
\begin{equation}\label{eq:Ck-lower}
 C_k\ge-aR_kS_k\ge-\frac1{a+1}.
\end{equation}

For the upper bound, both $R_k$ and $E_k:=1-aS_k+T_k$ decrease.
If $E_k\le0$, then $C_k\le0$. Otherwise,
\[
 C_k=R_kE_k\le R_1E_1
 =\frac{2(a^2+a+1)}{(a+1)^2(a+2)}\le\frac1{a+1},
\]
where the last inequality is equivalent to $a^2\le a$.
Together with \eqref{eq:Ck-lower}, this proves the result for $0<a<1$.
Each $\beta_k$ is continuous for $a>0$, so the bound also holds at $a=1$.
\end{proof}

We now conclude the proof of Theorem~\ref{thm:kernel} in the range $0<a\le1$.

\begin{proof}[Proof of Theorem~\ref{thm:kernel} for $0<a\le1$]
Let $|z|=\rho<1$. By Lemma~\ref{lem:small-coefficients},
\[
  \Rea H_a(z)
  \ge\frac1{a+1}
  \left(
    1-\sum_{k\ge1}\frac{\rho^k}{k(k+1)}
  \right)>0,
\]
because $\sum_{k\ge1}1/[k(k+1)]=1$ and $\rho^k<1$. This proves \eqref{eq:kernel-positivity} in the stated range.
\end{proof}
\section{Kernel positivity for \texorpdfstring{$a\ge1$}{a>=1}}\label{sec:large}

Throughout this section, $a\ge1$ is fixed. Let
\[
\mathrm B(x,y)
:=\int_0^1 t^{x-1}(1-t)^{y-1}\,\dd t
=\frac{\Gamma(x)\Gamma(y)}{\Gamma(x+y)},
\qquad x,y>0,
\]
denote the Euler beta function. Set
\begin{equation}\label{eq:beta-density}
	Z_a
	:=\mathrm B\!\left(a+\frac12,\frac12\right)
	=\frac{\sqrt\pi\,\Gamma(a+1/2)}{\Gamma(a+1)},
	\qquad
	\omega_a(x)
	:=\frac{x^{a-1/2}(1-x)^{-1/2}}{Z_a},
	\quad 0<x<1.
\end{equation}
We use the Pochhammer symbol $(x)_m:=\Gamma(x+m)/\Gamma(x)$ for $m\ge0$. Define also
\begin{equation}\label{eq:mean-log}
  m_a:=\int_0^1\omega_a(x)\log x\,\dd x
  <0.
\end{equation}
Thus $\omega_a$ is a probability density. Define its centered primitive and the associated first-order expression by
\begin{equation}\label{eq:Wa-Ra}
  W_a(x):=-\int_0^x\omega_a(u)(\log u-m_a)\,\dd u,
  \qquad
  R_a(x):=aW_a(x)-xW_a'(x).
\end{equation}
The definition gives $W_a(0)=0$, while the centering in
\eqref{eq:mean-log} gives $W_a(1)=0$.

For $0<c<1$, put $d:=1-c$ and define
\begin{equation}\label{eq:Tc-definition}
  \mathcal T_c f
  :=\PV\int_0^1\frac{f(x)}{x-c}\,\dd x
  +\pi\sqrt{\frac cd}\,f(c),
\end{equation}
whenever the principal value exists. The main task of this section is to prove
\[
  \mathcal T_c R_a>0,
  \qquad 0<c<1.
\]

We will use the following elementary boundary formula. It also specifies
the sign convention for all principal values below.

\begin{lem}\label{lem:cauchy-boundary}
Let $f\in L^1(0,1)$ be differentiable at $c\in(0,1)$. Then
\begin{equation}\label{eq:PV-subtraction}
 \PV\int_0^1\frac{f(x)}{x-c}\,\dd x
 =\int_0^1\frac{f(x)-f(c)}{x-c}\,\dd x
  +f(c)\log\frac{1-c}{c}.
\end{equation}
If $z\to c$ with $\Im z<0$ and $|\Rea z-c|\le-\Im z$, then
\begin{equation}\label{eq:cauchy-boundary}
 \int_0^1\frac{f(x)}{x-z}\,\dd x
 \longrightarrow\PV\int_0^1\frac{f(x)}{x-c}\,\dd x-i\pi f(c).
\end{equation}
\end{lem}

\begin{proof}
The difference quotient $q_c(x)=(f(x)-f(c))/(x-c)$ is integrable:
it is bounded near $c$, and its denominator is bounded away from zero
elsewhere. Subtracting $f(c)$ under a symmetric cutoff gives
\eqref{eq:PV-subtraction}, since the cutoff integral of $1/(x-c)$ is
$\log((1-c)/c)$. For the boundary limit, write
\[
 \int_0^1\frac{f(x)}{x-z}\,\dd x
 =\int_0^1\frac{x-c}{x-z}q_c(x)\,\dd x
  +f(c)\bigl(\operatorname{Log}(1-z)-\operatorname{Log}(-z)\bigr),
\]
using principal logarithms. In the stated cone, $|(x-c)/(x-z)|\le2$,
so dominated convergence applies to the first term. The logarithmic
difference tends to $\log((1-c)/c)-i\pi$, proving the formula.
\end{proof}

\subsection{Boundary values of the kernel}

Define
\begin{equation}\label{eq:G-gamma}
  G_a(t):=\sum_{k\ge1}\frac{\dot\alpha_k(a)}{k}t^k,
  \qquad
  \gamma(t):=\frac{(1+t)^2}{4t},
  \qquad 0<|t|<1,
\end{equation}
where $\alpha_k(a)$ is given by \eqref{eq:alpha-definition}. The Joukowski map converts the kernel into a Stieltjes transform with controlled boundary values.

\begin{prop}\label{prop:boundary-representation}
For $a\ge1$, the kernel has the representation
\begin{equation}\label{eq:H-Stieltjes}
  H_a(t)
  =\frac{t-1}{2t}\int_0^1\frac{R_a(x)}{x-\gamma(t)}\,\dd x,
  \qquad 0<|t|<1.
\end{equation}
If $0<\theta<\pi$ and
$
  c:=\cos^2\frac\theta2,
  d:=\sin^2\frac\theta2,
$
then
\begin{equation}\label{eq:boundary-H}
  \Rea H_a(\e^{i\theta})=d\,\mathcal T_c R_a,
\end{equation}
where $\mathcal T_c$ is defined by \eqref{eq:Tc-definition}.
\end{prop}

\begin{proof}
The substitution $x=\cos^2(\phi/2)$ in \eqref{eq:alpha-definition} gives
\[
 \alpha_k(a)=\int_0^1\omega_a(x)\,\mathrm T_k(2x-1)\,\dd x,
\]
where $\mathrm T_k$ is a Chebyshev polynomial. Differentiation in $a$ is
justified locally by the integrable bound
$C x^{a/2-1/2}(1-x)^{-1/2}(1+|\log x|)$.
Since $W_a'=-\omega_a(\log x-m_a)$ and $W_a(0)=W_a(1)=0$,
integration by parts yields
\begin{equation}\label{eq:alpha-dot-W}
 \dot\alpha_k(a)=2k\int_0^1W_a(x)\,\mathrm U_{k-1}(2x-1)\,\dd x.
\end{equation}
Here $\mathrm U_j$ is the Chebyshev polynomial of the second kind.
The bounds $|\mathrm U_j(s)|\le j+1$ for $|s|\le1$ and
$\sum_{j\ge0}(j+1)|t|^{j+1}<\infty$ justify summation under the integral.
The generating function $\sum_{j\ge0}\mathrm U_j(s)t^j=(1-2st+t^2)^{-1}$
\cite[p.~449, Eq.~(18.12.10)]{OLBC10} therefore gives
\begin{equation}\label{eq:G-Stieltjes}
 G_a(t)=\frac12\int_0^1\frac{W_a(x)}{\gamma(t)-x}\,\dd x.
\end{equation}

Write $C_f(w)=\int_0^1 f(x)/(x-w)\,\dd x$. The map $\gamma$ sends
$0<|t|<1$ outside $[0,1]$. Integration by parts gives
\[
 C_{R_a}(w)=aC_{W_a}(w)-wC_{W_a}'(w).
\]
Also, the series defining $G_a$ and $H_a$ give
\[
 H_a(t)=(1+t)G_a'(t)+a(t^{-1}-1)G_a(t).
\]
Substituting $G_a=-C_{W_a}\circ\gamma/2$ and using
$-(1+t)\gamma'(t)=(1-t)\gamma(t)/t$ proves \eqref{eq:H-Stieltjes}.

For the boundary limit, put $t=r\e^{i\theta}$, $0<\theta<\pi$.
With $c=\cos^2(\theta/2)$, one has
\[
 \Rea\gamma(t)-c=\frac{(1-r)^2\cos\theta}{4r},\qquad
 -\Im\gamma(t)=\frac{(1-r^2)\sin\theta}{4r}.
\]
Thus $\gamma(t)$ eventually lies in the cone of
Lemma~\ref{lem:cauchy-boundary}. Moreover, $R_a\in L^1(0,1)$ because
$W_a$ is bounded and $W_a'\in L^1(0,1)$; it is smooth near $c$.
The lemma consequently gives
\[
 C_{R_a}(\gamma(t))\longrightarrow
 \PV\int_0^1\frac{R_a(x)}{x-c}\,\dd x-i\pi R_a(c).
\]
Since $(\e^{i\theta}-1)/(2\e^{i\theta})=d+i\sqrt{cd}$, taking real
parts in \eqref{eq:H-Stieltjes} proves \eqref{eq:boundary-H}.
\end{proof}

\subsection{The beta transform}

Define
\begin{equation}\label{eq:L-kappa}
  L_a(c):=(1-c)\mathcal T_c\omega_a,
  \qquad
  \kappa_a:=\frac\pi{Z_a}
  =\frac{\sqrt\pi\,\Gamma(a+1)}{\Gamma(a+1/2)}.
\end{equation}

The transform of the basic beta density is explicitly positive.

\begin{lem}\label{lem:beta-transform}
If $a\ge1$, $0<c<1$, and $d=1-c$, then
\begin{equation}\label{eq:L-positive}
  L_a(c)
  =\kappa_a c^a
   +2ad\int_0^1(c+dt^2)^{a-1}\,\dd t>0.
\end{equation}
Moreover,
\begin{equation}\label{eq:kappa-bound}
  \kappa_a\le2a,
\end{equation}
with equality only at $a=1$.
\end{lem}

\begin{proof}
For $0<u<1$, let
\[
 f_u(x)=\frac{\mathbf1_{(u,1)}(x)}{\pi\sqrt{(x-u)(1-x)}}.
\]
The beta integral gives the mixture identity
\[
 \int_0^1au^{a-1}f_u(x)\,\dd u=\omega_a(x).
\]
For $\Im z<0$, the substitution $x=u+(1-u)\sin^2v$ gives
\[
 A_u(z):=\int_u^1\frac{f_u(x)}{x-z}\,\dd x
 =\frac1{\sqrt{u-z}\sqrt{1-z}},
\]
where both square roots are principal. The identity follows first for
real $z<u$ by the elementary tangent substitution and then by analytic
continuation. Absolute integrability off the real axis permits Fubini:
\[
 \int_0^1\frac{\omega_a(x)}{x-z}\,\dd x
 =a\int_0^1u^{a-1}A_u(z)\,\dd u.
\]
For $z=c-i\varepsilon$, the real part of $A_u(z)$ tends to $0$ for $u<c$
and to $((u-c)d)^{-1/2}$ for $u>c$. Furthermore,
\[
 |A_u(c-i\varepsilon)|\le\frac1{\sqrt d\sqrt{|u-c|}},\qquad
 au^{a-1}\le a.
\]
This is an integrable majorant. Dominated convergence and
Lemma~\ref{lem:cauchy-boundary}, applied to $\omega_a$, yield
\[
 \PV\int_0^1\frac{\omega_a(x)}{x-c}\,\dd x
 =a\int_c^1\frac{u^{a-1}}{\sqrt{(u-c)d}}\,\dd u.
\]
The residue term in $d\mathcal T_c\omega_a$ is
$\pi\sqrt{cd}\,\omega_a(c)=\kappa_a c^a$.
Substituting $u=c+dt^2$ in the remaining integral proves
\eqref{eq:L-positive}. Finally,
\[
 \kappa_a=a\,\mathrm B(a,1/2)
 =a\int_0^1u^{a-1}(1-u)^{-1/2}\,\dd u\le2a,
\]
with equality exactly when $a=1$.
\end{proof}

\subsection{A positive Abel decomposition}

For every integer $m\ge0$, define the beta moments
\[
  \nu_m(a)
  :=\int_0^1\omega_a(x)(1-x)^m\,\dd x
  =\frac{(1/2)_m}{(a+1)_m}.
\]
They satisfy
\begin{equation}\label{eq:nu-recurrence}
  (a+m)\nu_m(a)
  =\left(m-\frac12\right)\nu_{m-1}(a),
  \qquad m\ge1.
\end{equation}

The function $R_a$ need not have a fixed sign. The following positive Abel decomposition resolves it into elementary pieces whose transforms can be controlled individually.

\begin{lem}\label{lem:abel-decomposition}
For every $a\ge1$, there are numbers
\begin{equation}\label{eq:delta-properties}
  \delta_N(a)>0,
  \qquad
  \sum_{N\ge0}\delta_N(a)<\infty,
\end{equation}
such that
\begin{equation}\label{eq:W-basis}
  W_a(x)=\frac1{Z_a}
  \sum_{N\ge0}\delta_N(a)F_{a,N}(x),
\end{equation}
where
\begin{equation}\label{eq:F-basis}
  F_{a,N}(x)
  :=x^{a-1/2}(1-x)^{1/2}
    \bigl(1-(1-x)^{N+1}\bigr).
\end{equation}
If
\begin{equation}\label{eq:R-basis-def}
  R_{a,N}(x):=aF_{a,N}(x)-xF_{a,N}'(x),
\end{equation}
then
\begin{equation}\label{eq:R-basis}
  R_a(x)=\frac1{Z_a}
  \sum_{N\ge0}\delta_N(a)R_{a,N}(x).
\end{equation}
The series in \eqref{eq:R-basis} converges absolutely in $L^1(0,1)$. The series in \eqref{eq:W-basis} and \eqref{eq:R-basis}, together with their termwise derivatives, converge locally uniformly on $(0,1)$.
\end{lem}

\begin{proof}
Suppress $a$ in the beta moments and define
\begin{equation}\label{eq:t-tail}
 S_j:=\sum_{k\ge j+1}\frac{\nu_k}{k},\qquad
 t_j:=\frac{S_j}{(j+1/2)\nu_j},\qquad j\ge0.
\end{equation}
The recurrence \eqref{eq:nu-recurrence} gives $\nu_j\le1/(j+1)$, so
these sums are finite. It also gives
\[
 \frac{\nu_{j+1}}{j+1}
 \le\frac{\nu_j-\nu_{j+1}}{a+1/2}.
\]
Summing this telescoping inequality proves
\begin{equation}\label{eq:t-bound}
 0<t_j\le\frac1{(j+1/2)(a+1/2)},\qquad t_j\longrightarrow0.
\end{equation}
By the nonnegative series $-\log x=\sum_{k\ge1}(1-x)^k/k$,
Tonelli's theorem gives $t_0=-2m_a$.
Separating the first term of $S_j$ and using the moment recurrence gives
\begin{equation}\label{eq:t-backward}
 t_j=\frac{j+3/2}{a+j+1}t_{j+1}
       +\frac1{(j+1)(a+j+1)}.
\end{equation}
Consequently, the numbers
\begin{equation}\label{eq:delta-definition}
 \delta_j(a):=t_j-t_{j+1}
 =\frac1{(j+1)(a+j+1)}-\frac{a-1/2}{a+j+1}t_{j+1}
\end{equation}
are positive: \eqref{eq:t-bound} implies
$(a-1/2)t_{j+1}\le(a-1/2)/((j+3/2)(a+1/2))<1/(j+1)$.
Also, $\sum_j\delta_j(a)=t_0<\infty$.

Let $\tau_a(y)=\sum_{j\ge0}t_jy^j$, analytic for $|y|<1$.
The recurrence \eqref{eq:t-backward} and $t_0=-2m_a$ give
\[
 y(1-y)\tau_a'(y)+\bigl(1/2-(a+1)y\bigr)\tau_a(y)
 =\log(1-y)-m_a.
\]
It follows by differentiation that
\begin{equation}\label{eq:W-tau}
 W_a(x)=\frac1{Z_a}x^{a+1/2}(1-x)^{1/2}\tau_a(1-x):
\end{equation}
both sides have derivative $-\omega_a(x)(\log x-m_a)$ and vanish at $x=1$.
Since $t_j=\sum_{N\ge j}\delta_N(a)$, Tonelli's theorem gives
\begin{equation}\label{eq:Abel-tau}
 \tau_a(y)=\sum_{N\ge0}\delta_N(a)\sum_{j=0}^Ny^j,
 \qquad 0<y<1.
\end{equation}
The finite geometric sum now proves \eqref{eq:W-basis}.

Direct differentiation gives
\begin{equation}\label{eq:R-basis-explicit}
 R_{a,N}(x)=x^{a-1/2}(1-x)^{-1/2}r_N(1-x),
\end{equation}
where
\begin{equation}\label{eq:rN}
 \begin{split}
 r_N(y)&=\frac12-\left(N+\frac32\right)y^{N+1}+(N+1)y^{N+2}\\
       &=\frac12(1-y^{N+1})-(N+1)(1-y)y^{N+1}.
 \end{split}
\end{equation}
Thus $|r_N|\le3/2$ on $[0,1]$, and the summability of $\delta_N(a)$
proves absolute $L^1$ convergence of the series in \eqref{eq:R-basis}.
On every compact subinterval of $(0,1)$, the quantities
$(N+2)^2(1-x)^N$ are bounded uniformly in $N$. The formulas for
$F_{a,N}$ and $R_{a,N}$ therefore bound them and their derivatives uniformly
there. The Weierstrass test justifies local termwise differentiation,
proving \eqref{eq:R-basis} and all asserted convergence properties.
\end{proof}

\subsection{The finite-basis inequality}

We next prove that the transform is positive on every basis element in the Abel decomposition.

\begin{lem}\label{lem:finite-basis}
For $a\ge1$, an integer $N\ge0$, and $0<c<1$, set
\begin{equation}\label{eq:J-definition}
  \mathcal J_{a,N}(c)
  :=\frac1{Z_a}\mathcal T_c R_{a,N}.
\end{equation}
Then
\begin{equation}\label{eq:J-positive}
  \mathcal J_{a,N}(c)>0.
\end{equation}
\end{lem}

\begin{proof}
Put $d=1-c$ and suppress $a$ in $\nu_m$ and $L_a(c)$ when convenient.
Polynomial division gives
\begin{equation}\label{eq:power-transform}
 \mathcal T_c[\omega_a(x)(1-x)^m]
 =d^{m-1}L_a(c)-\sum_{k=0}^{m-1}d^k\nu_{m-1-k},
\end{equation}
with the sum empty for $m=0$. Substituting \eqref{eq:rN} and applying
\eqref{eq:nu-recurrence} to the finite sums yields
\[
 \mathcal J_{a,N}(c)=L_a(c)C_N(c)+aU_N(a,c)+cV_N(a,c),
\]
where
\begin{align}
 C_N(c)&=\frac1{2d}-\left(\frac12+c(N+1)\right)d^N,
 \label{eq:C-N}\\
 U_N(a,c)&=\sum_{k=0}^Nd^k\nu_{N+1-k},\qquad
 V_N(a,c)=\sum_{k=1}^{N+1}k d^{k-1}\nu_{N+1-k}.
 \label{eq:U-V-N}
\end{align}

We first record two bounds used in both parts of the induction.
Since $t^2\le t$ on $[0,1]$, Lemma~\ref{lem:beta-transform} gives
\begin{equation}\label{eq:L-global-upper}
 L_a(c)\le2ac^a+2ad\int_0^1(c+dt)^{a-1}\,\dd t
 =2+2(a-1)c^a.
\end{equation}
If $0<c<1/2$ and $b>0$, then $\log c<c-1<-1/2$ and
$b\le\e^{b/2}<c^{-b}$, where the first inequality follows from
$2u\le\e^u$. Including $b=0$, we obtain $bc^b<1$. Hence
\begin{equation}\label{eq:a-c-bound}
 (a-1)c^a<c,\qquad L_a(c)<2+2c\quad(0<c<1/2).
\end{equation}

\smallskip\noindent\emph{Base case.}
Taking $N=0$ in the preceding finite sums gives
\begin{equation}\label{eq:J0}
 \mathcal J_{a,0}(c)=\frac{c(2c-1)}{2d}L_a(c)+a\nu_1+c.
\end{equation}
For $c\ge1/2$, this is positive. For $c<1/2$, the coefficient of $L_a(c)$
is negative, so \eqref{eq:a-c-bound} gives
\begin{equation}\label{eq:J0-positive}
 \mathcal J_{a,0}(c)>\frac{2c^3}{d}+a\nu_1>0.
\end{equation}

\smallskip\noindent\emph{Induction step.}
Set $m=N+1$ and $\Sigma_m=\sum_{k=0}^m d^k\nu_{m-k}\ge d^m$.
Subtracting the finite sums in \eqref{eq:C-N}--\eqref{eq:U-V-N} gives
\begin{equation}\label{eq:J-recurrence}
 \mathcal J_{a,N+1}(c)-d\mathcal J_{a,N}(c)
 =a\nu_{m+1}+c\Sigma_m+cL_a(c)\left(\frac1{2d}-d^m\right).
\end{equation}
If $\xi:=d^{m+1}\le1/2$, the right side is positive.
If $\xi>1/2$, put $y=(1-\xi)/\xi\in(0,1)$ and
$\rho=1-1/(2\xi)=(1-y)/2$.
Since $\xi\le d$, we have $c<1/2$ and $y\ge c/d>c$.
Equations \eqref{eq:L-global-upper}--\eqref{eq:a-c-bound} imply
\[
 L_a(c)\rho<(2+2y)\frac{1-y}{2}=1-y^2<1.
\]
Using $\Sigma_m\ge d^m$ in \eqref{eq:J-recurrence}, we conclude that
\[
 \mathcal J_{a,N+1}(c)-d\mathcal J_{a,N}(c)
 \ge a\nu_{m+1}+cd^m\bigl(1-L_a(c)\rho\bigr)>0.
\]
Together with the base case, this proves the lemma by induction.
\end{proof}

\subsection{Passage to the principal value}

The preceding estimates allow us to pass to the principal value and complete the range $a\ge1$.

\begin{prop}\label{prop:large-positivity}
For every $a\ge1$ and $0<c<1$,
\begin{equation}\label{eq:T-sum}
  \mathcal T_c R_a
  =\sum_{N\ge0}\delta_N(a)\mathcal J_{a,N}(c)>0.
\end{equation}
Consequently,
\begin{equation}\label{eq:large-kernel-positive}
  \Rea H_a(z)>0,
  \qquad z\in\D.
\end{equation}
\end{prop}

\begin{proof}
Write $p_N(x)=r_N(1-x)$, so that $R_{a,N}/Z_a=\omega_a p_N$.
We have $|p_N|\le3/2$ on $[0,1]$. Their derivatives are uniformly bounded
on $[c/2,1]$, because $(N+2)^2(1-c/2)^N$ is bounded. The mean-value theorem
on that interval, and $|x-c|\ge c/2$ on $(0,c/2)$, therefore give a constant
$K_c$ independent of $N$ such that
\[
 \left|\frac{p_N(x)-p_N(c)}{x-c}\right|\le K_c.
\]
The identity
\[
 \frac{\omega_a(x)p_N(x)-\omega_a(c)p_N(c)}{x-c}
 =p_N(x)\frac{\omega_a(x)-\omega_a(c)}{x-c}
  +\omega_a(c)\frac{p_N(x)-p_N(c)}{x-c}
\]
gives the integrable majorant
$\frac32|\omega_a(x)-\omega_a(c)|/|x-c|+\omega_a(c)K_c$.
Since $\sum_N\delta_N(a)<\infty$, Fubini applies to the difference-quotient
series in \eqref{eq:PV-subtraction}; local uniform convergence handles
the value and residue terms. Thus
\[
 \mathcal T_c R_a
 =\frac1{Z_a}\sum_{N\ge0}\delta_N(a)\mathcal T_cR_{a,N}
 =\sum_{N\ge0}\delta_N(a)\mathcal J_{a,N}(c)>0
\]
by Lemmas~\ref{lem:abel-decomposition} and~\ref{lem:finite-basis}.

Proposition~\ref{prop:boundary-representation} gives $\Rea H_a>0$ on the
upper open semicircle; the real Taylor coefficients give the same on
the lower one. Continuity on $\overline\D$ gives nonnegative values at
$1$ and $-1$. The harmonic minimum principle now gives $\Rea H_a\ge0$
in $\D$. Equality at an interior point would force $\Rea H_a$ to be
identically zero, contradicting $H_a(0)=1/(a+1)>0$.
\end{proof}

Combining Proposition~\ref{prop:large-positivity} with Section~\ref{sec:small} completes the proof of Theorem~\ref{thm:kernel}.
\section{Completion of the proof}\label{sec:completion}

We first assemble the preceding reductions under the simple-zero hypothesis, and then remove that hypothesis by approximation.

\begin{prop}\label{prop:positive-exponents}
Let $p$ satisfy the normalization \eqref{eq:normalized-p}, assume that its zeros on $\T$ are simple, and set $Q_n(z):=1+z^n$. Then
\begin{equation}\label{eq:Fp-definition}
  F_p(r):=\log\frac{\norm{p}_r}{\norm{Q_n}_r}
\end{equation}
is nondecreasing for $r>0$.
\end{prop}

\begin{proof}
Fix $r=2a>0$. Theorem~\ref{thm:kernel} and the area identity \eqref{eq:area-identity} give
$\mathcal P_{2a}(q,B)\ge0$. Proposition~\ref{prop:entropy-reduction} therefore yields
\[
  \mathsf D(\rho_{p,r}\Vert m)
  \ge\mathsf D(\rho_{Q_n,r}\Vert m).
\]
Lemma~\ref{lem:entropy-derivative} now gives $F_p'(r)\ge0$.
\end{proof}

To remove the simple-zero hypothesis, we perturb the arguments of the roots while preserving their product.

\begin{lem}\label{lem:simple-approximation}
Let $p$ satisfy \eqref{eq:normalized-p}, with arbitrary multiplicities among its zeros on $\T$. Then there are polynomials $p_\nu$ satisfying the same normalization, each with simple zeros on $\T$, such that $p_\nu\to p$ coefficientwise and uniformly on $\T$.
\end{lem}

\begin{proof}
List the roots as $\e^{i\theta_1},\ldots,\e^{i\theta_n}$, choosing the same
argument for repeated roots. Set $v_j=j-(n+1)/2$, so that $\sum_jv_j=0$,
and define
\[
 p_\varepsilon(z)=\prod_{j=1}^n
 \left(z-\e^{i(\theta_j+\varepsilon v_j)}\right).
\]
For sufficiently small $\varepsilon>0$, each pair of distinct original
roots remains distinct by continuity. A repeated pair separates because
$0<\varepsilon|v_j-v_k|<2\pi$. There are only finitely many pairs, so a
single $\varepsilon_0>0$ makes all roots simple for $0<\varepsilon<\varepsilon_0$.
The identity $\sum_jv_j=0$ preserves the product of the roots and hence the
constant coefficient. Thus $p_\varepsilon$ is monic with constant coefficient
$1$ and unimodular zeros, which also gives $p_\varepsilon=p_\varepsilon^*$.
Its coefficients tend to those of $p$; their absolute differences bound
$\sup_{\overline\D}|p_\varepsilon-p|$. Taking any
$\varepsilon_\nu\downarrow0$ below $\varepsilon_0$ proves the lemma.
\end{proof}

We can now prove Theorem~\ref{thm:main}.

\begin{proof}
First let $p$ satisfy \eqref{eq:normalized-p}, and take the approximants
$p_\nu$ from Lemma~\ref{lem:simple-approximation}. For every fixed $r>0$,
uniform convergence on $\T$ and the bound $|p_\nu|\le2^n$ give
$\norm{p_\nu}_r\to\norm p_r$ by dominated convergence. Passing to the limit
in Proposition~\ref{prop:positive-exponents} therefore proves
\[
 \frac{\norm p_s}{\norm{Q_n}_s}\le
 \frac{\norm p_t}{\norm{Q_n}_t},\qquad 0<s\le t<\infty.
\]

For the endpoint $0$, let $X=\log|p|$ off the finite zero set, and set
$X=0$ on that null set. The local logarithmic singularities are integrable,
so $X\in L^1(\T,m)$ and $\e^{rX}=|p|^r$ almost everywhere.
With $M=\max\{1,\norm p_\infty\}$, the mean-value theorem gives
\[
 \left|\frac{\e^{rX}-1}{r}\right|\le M|X|,\qquad 0<r\le1.
\]
Dominated convergence yields
$M_p(r)=1+r\int X\,\dd m+o(r)$; taking logarithms and dividing by $r$ gives
$\norm p_r\to\exp\int X\,\dd m=\norm p_0$.
The same holds for $Q_n$, and Jensen's formula gives
$\norm p_0=\norm{Q_n}_0=1$.

For the endpoint $\infty$, put $S=\max_{\T}|p|$. If $0<b<S$, continuity
gives $m\{|p|>b\}>0$, and
\[
 b\,m\{|p|>b\}^{1/r}\le\norm p_r\le S.
\]
Letting first $r\to\infty$ and then $b\uparrow S$ proves
$\norm p_r\to S=\norm p_\infty$; the same applies to $Q_n$.
The corresponding one-sided limits therefore extend the comparison to
every $0\le s\le t\le\infty$.

Finally, Lemma~\ref{lem:normalization} writes the normalized polynomial as
$p(z)=cP(\eta z)$, where $c\ne0$ and $|\eta|=1$. Rotation preserves every
mean and multiplication by $c$ scales them all by $|c|$, so the comparison
also holds for the original polynomial $P$.
\end{proof}
\section{Consequences and neighboring extremal problems}\label{sec:applications}

For $0<r<\infty$, set
\begin{equation}\label{eq:Cr-definition}
  C_r:=\norm{1+z}_r,
  \qquad
  C_0:=1,
  \qquad
  C_\infty:=2.
\end{equation}
Since $z\mapsto z^n$ preserves the normalized Haar measure $m$, one has $\norm{Q_n}_r=C_r$. The beta integral gives
\[
  C_r^r
  =\frac1{2\pi}\int_0^{2\pi}|1+\e^{i\theta}|^r\,\dd\theta
  =\frac{\Gamma(r+1)}{\Gamma(1+r/2)^2}.
\]
In particular,
\[
  C_1=\frac4\pi,
  \qquad C_2=\sqrt2,
  \qquad C_4=6^{1/4}.
\]

\subsection{Reverse norm and coefficient inequalities}

Theorem~\ref{thm:main} first yields the following sharp family of norm and coefficient inequalities.

\begin{cor}\label{cor:applications}
Let
\begin{equation}\label{eq:p-coefficients}
  p(z)=\sum_{k=0}^na_kz^k
  =a_n\prod_{j=1}^n(z-\zeta_j),
  \qquad |\zeta_j|=1,
\end{equation}
be a nonzero polynomial of degree $n\ge1$. Let $C_r$ be defined by \eqref{eq:Cr-definition}.

\begin{enumerate}[label=\textnormal{(\roman*)},leftmargin=*,itemsep=1ex]
\item For every $0\le s\le t\le\infty$,
\begin{equation}\label{eq:reverse-norm}
  \norm{p}_t\ge\frac{C_t}{C_s}\norm{p}_s.
\end{equation}
The constant $C_t/C_s$ is sharp.

\item For every $0\le r\le\infty$,
\begin{equation}\label{eq:visser-Lr}
  |a_0|+|a_n|
  \le\frac2{C_r}\norm{p}_r.
\end{equation}
Thus the constants for $r=1,2,\infty$ are $\pi/2$, $\sqrt2$, and $1$, respectively.

\item One has
\begin{equation}\label{eq:OHR}
  \norm{p}_\infty^2
  \ge2\norm{p}_2^2
  =2\sum_{k=0}^n|a_k|^2.
\end{equation}
More generally, if $m$ is a positive integer and
$
  p(z)^m=\sum_{\ell=0}^{mn}b_\ell z^\ell,
$
then
\begin{equation}\label{eq:higher-power-OHR}
  \norm{p}_\infty^{2m}
  \ge\frac{4^m}{\binom{2m}{m}}
  \sum_{\ell=0}^{mn}|b_\ell|^2.
\end{equation}

\item For every integer $m\ge1$,
\begin{equation}\label{eq:even-moment}
  \int_{\T}|p|^{2m}\,\dd m
  \ge\frac{\binom{2m}{m}}{2^m}
  \left(\int_{\T}|p|^2\,\dd m\right)^m.
\end{equation}
For $m=2$, this is the convolution inequality
\begin{equation}\label{eq:convolution}
  \sum_{\ell=0}^{2n}
  \left|\sum_{\substack{u+v=\ell\\0\le u,v\le n}}a_ua_v\right|^2
  \ge\frac32\left(\sum_{k=0}^n|a_k|^2\right)^2.
\end{equation}
\end{enumerate}

The constants in \textnormal{(i)}--\textnormal{(iv)} are attained by
$p(z)=c(z^n-\omega)$ with $c\ne0$ and $|\omega|=1$.
\end{cor}

\begin{proof}
Equation \eqref{eq:reverse-norm} is Theorem~\ref{thm:main}, because $\norm{Q_n}_r=C_r$.

For \eqref{eq:visser-Lr}, Jensen's formula and \eqref{eq:p-coefficients} give
\[
  \norm{p}_0=|a_n|,
  \qquad |a_0|=|a_n|.
\]
Apply \eqref{eq:reverse-norm} with $(s,t)=(0,r)$. The case $r=\infty$ agrees with Visser's classical inequality \cite[p.~279, Theorem~3]{Vis45}.

Taking $(s,t)=(2,\infty)$ proves \eqref{eq:OHR}; this is the O'Hara--Rodriguez inequality \cite[p.~333, Corollary~1]{OR74}. Taking $(s,t)=(2m,\infty)$ and using
$C_{2m}^{2m}=\binom{2m}{m}$ gives \eqref{eq:higher-power-OHR}, because
\[
  \norm{p}_{2m}^{2m}=\norm{p^m}_2^2
  =\sum_{\ell=0}^{mn}|b_\ell|^2.
\]
Finally, $(s,t)=(2,2m)$ gives \eqref{eq:even-moment}; Parseval's identity applied to $p^2$ gives \eqref{eq:convolution}.

For $p(z)=c(z^n-\omega)$, rotation invariance gives $\norm{p}_r=|c|C_r$ for every $0\le r\le\infty$. Hence all constants asserted in the corollary are attained.
\end{proof}

\subsection{Erd\H{o}s--Szekeres products}

Applying the $L^2$-to-$L^\infty$ comparison to a product of cyclotomic factors gives the following consequence.

\begin{cor}\label{cor:erdos-szekeres}
Let $N\ge1$ and let $s_1,\ldots,s_N$ be positive integers. Define
$
  P_N(z):=\prod_{j=1}^N(1-z^{s_j})=\sum_k b_k z^k.
$
Then
\begin{equation}\label{eq:ES-energy}
  \norm{P_N}_\infty^2
  \ge2\sum_k |b_k|^2,
  \qquad
  \sum_k |b_k|^2\ge2N,
\end{equation}
and therefore
\begin{equation}\label{eq:ES-bound}
  \norm{P_N}_\infty\ge2\sqrt N.
\end{equation}
\end{cor}

\begin{proof}
Every zero of $P_N$ lies on $\T$, so the first inequality in \eqref{eq:ES-energy} is \eqref{eq:OHR}. Moreover, $zP_N(z)\in\mathbb Z[z]$ has zero constant term and is divisible by $(1-z)^N$. Tang's theorem \cite[pp.~3384--3385, Theorem~3.2]{Tan26}, applied to $zP_N$, gives \[ \sum_k\left|[z^k]\bigl(zP_N(z)\bigr)\right|^2\ge2N, \] where $[z^k]F(z)$ denotes the coefficient of $z^k$ in the polynomial $F(z)$. Multiplication by $z$ only shifts the coefficient sequence, so \[ \sum_k\left|[z^k]\bigl(zP_N(z)\bigr)\right|^2 =\sum_k|b_k|^2. \] Combining this with the first inequality in \eqref{eq:ES-energy} proves \eqref{eq:ES-bound}.
\end{proof}

\subsection{A mixed derivative inequality}

The extremal derivative identity for unimodular-zero polynomials combines with Theorem~\ref{thm:main} as follows.

\begin{cor}\label{cor:derivative}
Let $p$ be as in \eqref{eq:p-coefficients}, and let $C_r$ be defined by \eqref{eq:Cr-definition}. Then, for every $0\le r\le\infty$,
\begin{equation}\label{eq:mixed-derivative}
  \norm{p'}_\infty
  =\frac n2\norm{p}_\infty
  \ge\frac n{C_r}\norm{p}_r.
\end{equation}
The constant $n/C_r$ is sharp. In particular,
\[
  \norm{p'}_\infty\ge n|a_n|,
  \qquad
  \norm{p'}_\infty\ge\frac{\pi n}{4}\norm{p}_1,
  \qquad
  \norm{p'}_\infty\ge\frac n{\sqrt2}\norm{p}_2.
\]
\end{cor}

\begin{proof}
Since $p$ has no zero in $\D$, the Erd\H{o}s--Lax theorem gives
$\norm{p'}_\infty\le(n/2)\norm{p}_\infty$ \cite[pp.~509--513]{Lax44}. Since all zeros of $p$ lie in $\overline\D$, Tur\'an's reverse inequality gives the opposite bound \cite[p.~90, Satz~I]{Tur40}. Hence equality holds in the first part of \eqref{eq:mixed-derivative}; see also \cite[p.~332, Theorem~1]{OR74}. The second part follows from \eqref{eq:reverse-norm} with $t=\infty$. For $p(z)=c(z^n-\omega)$, both inequalities are equalities, so the constant is sharp.
\end{proof}

\subsection{Logarithmic variance and power sums}

For the zeros in \eqref{eq:p-coefficients}, define
\begin{equation}\label{eq:Sk-definition}
  S_k:=\sum_{j=1}^n\zeta_j^k,
  \qquad k\ge1.
\end{equation}

The behavior of the norm ratio at the zero exponent gives a sharp logarithmic variance bound.

\begin{cor}\label{cor:power-sums}
Let $p$ and $S_k$ be defined by \eqref{eq:p-coefficients} and \eqref{eq:Sk-definition}. Then
\begin{equation}\label{eq:log-variance}
  \int_{\T}\bigl(\log|p|-\log|a_n|\bigr)^2\,\dd m
  \ge\frac{\pi^2}{12},
\end{equation}
and equivalently
\begin{equation}\label{eq:power-sum}
  \sum_{k=1}^\infty\frac{|S_k|^2}{k^2}
  \ge\frac{\pi^2}{6}.
\end{equation}
Equality holds when the zeros form a regular $n$-gon.
\end{cor}

\begin{proof}
For a nonzero polynomial $f$, define
$Y_f(\zeta):=\log|f(\zeta)|$ when $f(\zeta)\ne0$, and set
$Y_f(\zeta):=0$ on its finite zero set. Then
$Y_f\in L^q(\T,m)$ for every finite $q$, its positive part is bounded,
and $\e^{rY_f}=|f|^r$ almost everywhere. On $\{Y_f<0\}$ and for $0<r\le1$,
\[
  0\le \e^{rY_f}-1-rY_f\le \frac{r^2Y_f^2}{2};
\]
on $\{Y_f\ge0\}$ the same remainder is bounded by $C_f r^2Y_f^2$, because $Y_f^+$ is bounded. The dominated convergence theorem therefore yields
\[
  \int_{\T}\e^{rY_f}\,\dd m
  =1+r\int_{\T}Y_f\,\dd m
   +\frac{r^2}{2}\int_{\T}Y_f^2\,\dd m+o(r^2).
\]
Taking the logarithm and dividing by $r$ gives
\begin{equation}\label{eq:small-r-expansion}
  \log\norm{f}_r
  =\int_{\T}Y_f\,\dd m
   +\frac r2\Var_m(Y_f)+o(r)
  \qquad(r\downarrow0),
\end{equation}
where $
\Var_m(Y)
:=\int_{\T}Y^2\,\dd m
-\left(\int_{\T}Y\,\dd m\right)^2
$
denotes the variance with respect to $m$.
Set
\[
F(r):=\log\norm{p}_r-\log\norm{Q_n}_r.
\]
By Theorem~\ref{thm:main}, $F(r)\ge F(0)$ for every $r>0$.
Equation~\eqref{eq:small-r-expansion} shows that the right derivative
$F'(0+)$ exists; dividing $F(r)-F(0)\ge0$ by $r$ and letting
$r\downarrow0$ gives $F'(0+)\ge0$. Hence
\[
  \Var_m(\log|p|)
  \ge\Var_m(\log|Q_n|).
\]
Jensen's formula gives $\int\log|p|\,\dd m=\log|a_n|$, while
\[
  \log|1+\e^{i\theta}|
  =\sum_{k=1}^\infty\frac{(-1)^{k-1}}k\cos(k\theta)
  \quad\text{in }L^2(0,2\pi).
\]
Because $z\mapsto z^n$ preserves the normalized Haar measure $m$, the variance of $\log|Q_n|$ equals that of $\log|1+z|$. Parseval's identity therefore gives
$\Var_m(\log|Q_n|)=\frac12\sum_{k\ge1}k^{-2}=\pi^2/12$, proving \eqref{eq:log-variance}.

For each unimodular $\zeta_j$,
\[
  \log|\e^{i\theta}-\zeta_j|
  =-\sum_{k=1}^\infty\frac1k
    \Rea\bigl(\overline{\zeta_j}^{\,k}\e^{ik\theta}\bigr)
  \quad\text{in }L^2(0,2\pi).
\]
Summing over $j$ and applying Parseval gives
\[
  \int_{\T}\bigl(\log|p|-\log|a_n|\bigr)^2\,\dd m
  =\frac12\sum_{k=1}^\infty\frac{|S_k|^2}{k^2}.
\]
Thus \eqref{eq:power-sum} is equivalent to \eqref{eq:log-variance}. If the zeros form a regular $n$-gon, then $S_k=0$ unless $n\mid k$, while $|S_{jn}|=n$; equality follows.
\end{proof}

\subsection{Relative entropy at every exponent}

Differentiating the monotone norm ratio at a positive exponent recovers the entropy comparison itself.

\begin{cor}\label{cor:entropy}
Let $p$ be as in \eqref{eq:p-coefficients}, and set $Q_n(z):=1+z^n$. For $r>0$ and $f\in\{p,Q_n\}$, define $\rho_{f,r}$ by \eqref{eq:moment-density}. Then
\begin{equation}\label{eq:entropy-corollary}
  \mathsf D(\rho_{p,r}\Vert m)
  \ge\mathsf D(\rho_{Q_n,r}\Vert m),
  \qquad r>0.
\end{equation}
Equality is attained by $p(z)=c(z^n-\omega)$ with $c\ne0$ and $|\omega|=1$.
\end{cor}

\begin{proof}
The function
$r\mapsto\log\norm{p}_r-\log\norm{Q_n}_r$ is differentiable for $r>0$ and nondecreasing by Theorem~\ref{thm:main}. Its derivative is therefore nonnegative. Lemma~\ref{lem:entropy-derivative} identifies $r^2$ times that derivative with the left-hand side minus the right-hand side of \eqref{eq:entropy-corollary}.

For $p(z)=c(z^n-\omega)$, multiplication by $c$ cancels in the normalized density. Choose $\eta\in\T$ with $\eta^n=-\omega$; then
$|(\eta z)^n-\omega|=|1+z^n|$. Rotation invariance therefore shows that $\rho_{p,r}$ and $\rho_{Q_n,r}$ have the same entropy, proving the equality assertion.
\end{proof}

%
%

\end{document}